\documentclass{gtmon_a}
\pdfoutput=1

\proceedingstitle{Groups, homotopy and configuration spaces (Tokyo
  2005)}
\conferencestart{5 July 2005}
\conferenceend{11 July 2005}
\conferencename{Groups, homotopy and configuration spaces,
                in honour of Fred Cohen's 60th birthday}
\conferencelocation{University of Tokyo, Japan}

\editor{Norio Iwase}
\givenname{Norio}
\surname{Iwase}

\editor{Toshitake Kohno}
\givenname{Toshitake}
\surname{Kohno}

\editor{Ran Levi}
\givenname{Ran}
\surname{Levi}

\editor{Dai Tamaki}
\givenname{Dai}
\surname{Tamaki}

\editor{Jie Wu}
\givenname{Jie}
\surname{Wu}

\title[On the Rothenberg--Steenrod spectral sequence]
{On the Rothenberg--Steenrod spectral sequence for 
the\\mod $2$ cohomology of classifying spaces of spinor groups}

\author{Masaki Kameko}
\givenname{Masaki}
\surname{Kameko}
\address{Department of Mathematics\\
Faculty of Regional Science\\
Toyama University of International Studies\\
65-1 Higashikuromaki\\
Toyama, 930-1292\\ Japan}
\email{kameko@tuins.ac.jp}
\urladdr{}

\author{Mamoru Mimura}
\givenname{Mamoru}
\surname{Mimura}
\address{Department of Mathematics\\
Faculty of Science\\
Okayama University\\
3-1-1 Tsushima-naka\\
Okayama, 700-8530\\ Japan}
\email{mimura@math.okayama-u.ac.jp}
\urladdr{}
\keyword{Steenrod square}
\keyword{classifying space}
\keyword{cohomology}
\keyword{Lie group}
\keyword{spectral sequence}
\subject{primary}{msc2000}{55R40}
\subject{secondary}{msc2000}{55T99}

\arxivreference{}

\volumenumber{13}
\issuenumber{}
\publicationyear{2008}
\papernumber{11}
\startpage{261}
\endpage{279}
\doi{}
\MR{}
\Zbl{}
\received{31 May 2006}
\revised{20 August 2007}
\accepted{27 August 2007}
\published{25 February 2008}
\publishedonline{25 February 2008}
\proposed{}
\seconded{}
\corresponding{}
\version{}

\makeatletter
\def\cnewtheorem#1[#2]#3{\newtheorem{#1}{#3}[section]
\expandafter\let\csname c@#1\endcsname\c@theorem}

\let\xysavmatrix\xymatrix
\def\xymatrix{\disablesubscriptcorrection\xysavmatrix}
\AtBeginDocument{\let\bar\wbar\let\tilde\wtilde\let\hat\what}
\let\qedhere\proved
\font\mathitt=cmmi10 scaled 1085
\font\mathits=cmmi7 scaled 1085

\newtheorem{theorem}{Theorem}[section]
\cnewtheorem{proposition}[theorem]{Proposition}
\cnewtheorem{lemma}[theorem]{Lemma}
\cnewtheorem{corollary}[theorem]{Corollary}
\theoremstyle{definition}
\cnewtheorem{definition}[theorem]{Definition}
\cnewtheorem{remark}[theorem]{Remark}

\makeatother  %  move after \newtheorem block

\newcommand{\spin}{\mathrm{Spin}}

\newcommand{\Sq}{{\mathrm{Sq}}}
\newcommand{\bZ}{\mathbb{Z}}
\newcommand{\bF}{\mathbb{F}}

\newcommand{\cotor}{\mathrm{Cotor}}

\newcommand{\lt}{\mathop{\mathrm{lt}}}

\newcommand{\lp}{\mathop{\mathrm{lp}}}
\newcommand{\lcm}{\mathop{\mathrm{lcm}}}
\newcommand{\groebner}{Gr\"{o}bner }
\newcommand{\poincare}{Poincar\'{e} }

\begin{document}

\begin{htmlabstract}
We compute the cotorsion product of the mod 2 cohomology of spinor
group spin(n), which is the E<sub>2</sub>&ndash;term of the
Rothenberg&ndash;Steenrod spectral sequence for the mod 2 cohomology of
the classifying space of the spinor group spin(n).  As a
consequence of this computation, we show the non-collapsing of the
Rothenberg&ndash;Steenrod spectral sequence for n&ge; 17.
\end{htmlabstract}

\begin{abstract} 
We compute the cotorsion product of the mod $2$ cohomology of spinor
group $\mathrm{spin}(n)$, which is the $E_2$--term of the
Rothenberg--Steenrod spectral sequence for the mod $2$ cohomology of
the classifying space of the spinor group $\mathrm{spin}(n)$.  As a
consequence of this computation, we show the non-collapsing of the
Rothenberg--Steenrod spectral sequence for $n\geq 17$.
\end{abstract}

\begin{asciiabstract}
We compute the cotorsion product of the mod 2 cohomology of spinor
group spin(n), which is the E_2-term of the Rothenberg-Steenrod
spectral sequence for the mod 2 cohomology of the classifying space
of the spinor group spin(n).  As a consequence of this computation,
we show the non-collapsing of the Rothenberg-Steenrod spectral
sequence for n > 16.
\end{asciiabstract}

\maketitle

%%%%%%%%%%%%%%%%%%%%   Start of main body of article

\section{Introduction}\label{sec1}

Let $n$ be a fixed integer greater than or equal to $9$. 
In \cite{quillen}, Quillen computed the mod $2$ cohomology of the classifying space $B\spin(n)$ using the Leray--Serre spectral sequence 
associated with the fiber bundle 
$B\pi\co B\spin(n) \to BSO(n)$.
In terms of  the Hurwitz--Radon number  $h$  given by
\[
\begin{array}{rl}
4\ell & \mbox{if $n=8\ell+1$,}\\
4\ell+1 & \mbox{if $n=8\ell+2$,}\\
4\ell+2 & \mbox{if $n=8\ell+3$ or $8\ell+4$,}\\
4\ell+3 &\mbox{if $n=8\ell+5$, $8\ell+6$, $8\ell+7$ or $8\ell+8$,}\\
\end{array}
\]
Quillen's result is stated as follows:
\begin{theorem}[Quillen]
As a graded $\bF_2$--algebra, we have
\[
H^{*}(B\spin(n);\bF_2)=
\bF_2[w_2,\ldots, w_n]/J \otimes \bF_2[z],
\]
where $J=(v_0, \ldots, v_{h-1})$, $v_0=w_2$, $v_k=\Sq^{2^{k-1}}\cdots \Sq^1 w_2$ for $1\leq k\leq h-1$ and $\deg z=2^{h}$.
Moreover, $v_0,\ldots, v_{h-1}$ is a regular sequence and the \poincare series is given by
\[
{\displaystyle \prod_{k=0}^{h-1} (1-t^{2^k+1})}\left\slash {\displaystyle\left\{ (1-t^{2^h}) \prod_{k=2}^{n} (1-t^k)\right\}}  .\right.
\]
\end{theorem}

On the other hand, the Rothenberg--Steenrod spectral sequence can often be the most powerful  tool for computing the mod $p$ cohomology of the classifying space $BG$
from the mod $p$ cohomology of the underlying connected compact Lie group $G$.
Its $E_2$--term is given by the cotorsion product 
\[
\cotor_{H^*(G;\bF_p)}(\bF_p, \bF_p)
\]
and it converges to the mod $p$ cohomology of the classifying space
$BG$.  Recently, we proved in \cite{kamekomimura} the non-degeneracy
of the Rothenberg--Steenrod spectral sequence for the mod $3$
cohomology of the classifying space $BE_8$ of the exceptional Lie
group $E_8$.  Until this paper all computational results in literature
indicated that the Rothenberg--Steenrod spectral sequence collapses at
the $E_2$--level.  Although it is not in literature, it has been a
folklore to experts for a long time that the Rothenberg--Steenrod
spectral sequence for the mod $2$ cohomology of the classifying space
$B\spin(n)$ does not collapse at the $E_2$--level for some $n$.  In
the case $n=2^{s-1}+1$, for example, it is easy to compute the
cotorsion product.  Since the mod $2$ cohomology of $\spin(2^{s-1}+1)$
is a primitively generated Hopf algebra, its cotorsion product is a
polynomial algebra $\bF_2[w_k]\otimes \bF_2[z']$ where $4\leq k\leq
2^{s-1}$, $k\not = 2^\ell+1$ ($\ell=1, \ldots, s-2$) and $\deg
z'=2^{s}$.  However, the mod $2$ cohomology of $B\spin(2^{s-1}+1)$ is
not a polynomial algebra for $s\geq 5$. So, comparing their \poincare
series, it is easy to deduce that the Rothenberg--Steenrod spectral
sequence does not collapse at the $E_2$--level.  In this paper,
through the computation of the cotorsion product
\[
\cotor_{H^*(\spin(n);\bF_2)}(\bF_2, \bF_2)
\]
for all $n\geq 9$, we give a proof for the non-degeneracy 
of the Rothenberg--Steenrod spectral sequence for all $n\geq 17$.

Let $s$ be an integer such that
\[
2^{s-1}<n \leq 2^{s}.
\]
In \fullref{sec2}, we define an integer $h'$ for $n\geq 9$. 
Using the integers $s$ and $h'$, our main result is stated as follows:
\begin{theorem} \label{sec12}
Let $A=H^{*}(\spin(n);\bF_2)$.
Suppose that $n\geq 9$.
Then, we have an isomorphism of graded $\bF_2$--algebras
\[
\cotor_{A}(\bF_2, \bF_2)=\bF_2[w_{2},\ldots, w_{n}]/J' \otimes \bF_2[z'],
\]
where 
$J'=(v_0, \ldots, v_{h'-1})$,  $v_0=w_2$,
\[
v_k=\underbrace{\Sq^0\cdots \Sq^0}_{\mbox{\rm $k$--times}} v_0\;\;\;\mbox{\rm ($k=1,\ldots, s-1$)},
\]
\[
v_s=\sum_{i+j=2^{s-1}} w_{2i+1}w_{2j}, 
\]
and
\[
v_{s+k}=\Sq^{2^{k-1}}\cdots \Sq^{1}v_s\;\;\;\mbox{\rm ($k\geq 1$).}
\] 
Moreover, the sequence $v_0, \ldots, v_{h'-1}$ is a regular sequence and 
the \poincare series of the cotorsion product is given by
\[
{\displaystyle \prod_{k=0}^{h'-1} (1-t^{2^{k}+1})}
\left\slash
{\displaystyle \left\{ (1-t^{2^{h'}}) \prod_{k=2}^{n} (1-t^{k})\right\}} .
\right.
\]
\end{theorem}
A caution is called for; the action of Steenrod squares in \fullref{sec12} is the one defined for the cotorsion product. 
It is not the one induced by the action of Steenrod squares on $A=H^*(\spin(n);\bF_2)$. In particular, $\Sq^0$ is not the identity homomorphism.
We recall the action of Steenrod square on the cotorsion product in \fullref{sec4}.
After defining the integer $h'$, we prove the following proposition in \fullref{sec2}.
\begin{proposition}\label{sec13}
For $9\leq n \leq 16$, we have $h'=h$.
For $n\geq 17$, we have $h'<h$. 
\end{proposition}
Thus, we have the following theorem.
\begin{theorem}
For $ n \leq 16$, the Rothenberg--Steenrod spectral sequence for the mod $2$ 
cohomology $H^{*}(B\spin(n);\bF_2)$ collapses at the $E_2$--level.
For $n\geq 17$, 
the Rothenberg--Steenrod spectral sequence for the mod $2$ 
cohomology $H^{*}(B\spin(n);\bF_2)$ does not collapse at the $E_2$--level.
\end{theorem}

The cotorsion products appear in other settings. There exist spectral sequences 
converging to the mod $p$ cohomology of classifying spaces of loop groups as well as to
 the one of classifying spaces of finite Chevalley groups.
Both spectral sequences have the same $E_2$--term:
\[
\cotor_{H^{*}(G;\bF_p)}(\bF_p, H^{*}(G;\bF_p)).
\]
In the case $G=\spin(10)$, $p=2$, the computation of the above
cotorsion product is done in Kuribayashi, Mimura and Nishimoto
\cite{kuribayashimimuranishimoto} using the twisted tensor product.
However, it seems to be not so easy to carry out their computation for
$n>10$.  In this paper, we use the change-of-rings spectral sequence
and Steenrod squares as our tools.  We hope that the computation done
in this paper can shed some light on the computation of the cotorsion
products
\[
\cotor_{H^{*}(G;\bF_p)}(\bF_p, H^{*}(G;\bF_p)).
\]
In \fullref{sec2}, we define integers $s$, $t$, $m$, $m'$,
$\varepsilon$, $h'$ and sets $C$, $D$, $E$ and prove some elementary
properties of these integers and sets as well as \fullref{sec13}.  We
use these integers and sets in order to describe generators and
relations of cotorsion products in \fullref{sec5}.  In \fullref{sec3},
we give a naive criterion for a sequence in a polynomial ring over a
field to be a regular sequence in terms of \groebner bases.  In
\fullref{sec4}, we recall some results on the Steenrod squares acting
on cotorsion products and the change-of-rings spectral sequence.  In
\fullref{sec5}, we prove \fullref{sec12} using the results in
Sections \ref{sec3} and \ref{sec4}.

We thank W Singer for showing us the manuscript of his book~\cite{singer}.
We also thank the referee for his/her careful reading of the manuscript.
The first named author was partially supported by Japan Society for the Promotion of Science, Grant-in-Aid for Scientific Research (C) 19540105
when preparing for the revised version of this paper.

\section{Integers $s$, $t$, $h'$}\label{sec2}

In this section, for a given integer $n\geq 9$, 
we define integers $s$, $t$, $m$, $m'$, $\varepsilon$, $h'$ and sets $C$, $D$, $E$ and
prove some elementary properties of these integers and sets.
We use these integers, sets and their properties in \fullref{sec5} in order to describe generators and relations, in particular  $v_{s+k}$ in Theorem 1.2, of cotorsion products. 
We do not use the results in this section until \fullref{sec5}.
Throughout  this section, we assume that $n$ is a fixed integer greater than or equal to $9$.

To begin with, we define integers $s$, $t$, $m$, $m'$ and $\varepsilon$.
For a positive integer $k$, let $\alpha(k)$  be the number of $1$'s in the binary expansion of $k$.
Let $s$ be an integer such that 
\[
2^{s-1}< n \leq 2^s.
\]
For $n<2^s-2$, let $t$ be an integer such that 
\[
2^{s}-2^t-1\leq n <2^s-2^{t-1}-1,
\]
and for  $n=2^s, 2^s-1, 2^s-2$, 
let $t=1$.

Let us consider a set of integers
\[
E=\{ k \in \bZ\;|\; 2\leq k \leq n, \alpha(k-1)\geq 2\},
\]
and its subset
\[
D=\{ k\in \bZ\;|\; k\leq n, 2^s-k+1\leq n, \alpha(k-1)\geq 2, \alpha(2^s-k)\geq 2 \}.
\]
It is easy to verify the following proposition.
\begin{proposition}
The set $D$ is empty if and only if  $n=2^{s-1}+1$.
\end{proposition}

\begin{proof}
Since $n \geq 9$, we may assume that $s\geq 4$.
Let $k=2^{s-1}+2$.
Then,  we have $\alpha(k-1)=2$ and  $\alpha(2^s-k)=s-2\geq 2$. 
Thus, if $n\geq  2^{s-1}+2$, we have $k\in D$.
If $n=2^{s-1}+1$ and $k' \in D$, then $2^{s}-(2^{s-1}+1)+1\leq k' \leq 2^{s-1}+1$. So, we have $k'=2^{s-1}$ or $2^{s-1}+1$.
Since $\alpha(2^s-2^{s-1})=1$ and $\alpha((2^{s-1}+1)-1)=1$, $2^{s-1}, 2^{s-1}+1\not\in D$. Therefore, $D$ is empty.
\end{proof}

When $D$ is not empty, let $m$ be the greatest integer in $D$, put
\[
m'=2^{s-t}(2^s-m)+1,
\]
and let us define $\varepsilon$ as follows: 
\[
\begin{array}{ll}
\varepsilon=0 & \mbox{if  $m'>n$,} \\
\varepsilon=1 & \mbox{if  $m'\leq n$.}
\end{array}
\]
We also define $h'$ as follows:
\[
\begin{array}{ll}
h'=s & \mbox{if $D=\emptyset$,} \\
h'=2s-t+\varepsilon & \mbox{if $D \not = \emptyset$.}
\end{array}
\]
Next, we prove \fullref{sec13} by computing $h'$ for $9\leq n \leq 32$ and 
by showing that the inequality $h'<h$ holds for $n\geq 33$. 

\begin{proof}[Proof of \fullref{sec13}]
For $n \leq 32$, by direct computation, we have the following tables.
\[
\begin{array}[t]{c|c|c|c|c|c|c|c|c}
n & s & t & m & m' & \varepsilon & h' & \ell & h \\ \hline
9 & 4 & 3 & - & - & - & 4 & 1 & 4 \\
10 & 4 & 3 & 10 & 13 & 0 & 5 & 1 & 5 \\
11 & 4 & 2 & 11 & 21 & 0 & 6 & 1 & 6 \\
12 & 4 & 2 & 11 & 21 & 0 & 6 & 1 & 6 \\
13 & 4 & 1 & 13 & 25 & 0 & 7 & 1 & 7 \\
14 & 4 & 1 & 13 & 25 & 0 & 7 & 1 & 7 \\
15 & 4 & 1 & 13 & 25 & 0 & 7 & 1 & 7 \\
16 & 4 & 1 & 13 & 25 & 0 & 7 & 1 & 7 
\end{array}\;\;\;
\begin{array}[t]{c|c|c|c|c|c|c|c|c}
n & s & t & m & m' & \varepsilon & h' & \ell & h \\ \hline
17 & 5 & 4 & - & - & - & 5 & 2 & 8 \\
18 & 5 & 4 & 18 & 29 & 0 & 6 & 2 & 9 \\
19 & 5 & 4 & 19 & 27 & 0 & 6 & 2 & 10 \\
20 & 5 & 4 & 20 & 25 & 0 & 6 & 2 & 10 \\
21 & 5 & 4 & 21 & 23 & 0 & 6 & 2 & 11 \\
22 & 5 & 4 & 22 & 21 & 1 & 7 & 2 & 11 \\
23 & 5 & 3 & 23 & 37 & 0 & 7 & 2 & 11 \\
24 & 5 & 3 & 23 & 37 & 0 & 7 & 2 & 11 \\
25 & 5 & 3 & 25 & 29 & 0 & 7 & 3 & 12 \\
26 & 5 & 3 & 26 & 25 & 1 & 8 & 3 & 13 \\
27 & 5 & 2 & 27 & 41 & 0 & 8 & 3 & 14 \\
28 & 5 & 2 & 27 & 41 & 0 & 8 & 3 & 14 \\
29 & 5 & 1 & 29 & 49 & 0 & 9 & 3 & 15 \\
30 & 5 & 1 & 29 & 49 & 0 & 9 & 3 & 15 \\
31 & 5 & 1 & 29 & 49 & 0 & 9 & 3 & 15 \\
32 & 5 & 1 & 29 & 49 & 0 & 9 & 3 & 15 
\end{array}
\]
Next, we deal with the case $n\geq 33$.  In this case, we may assume that $s\geq 6$. 
By the definition of $t$, we have $t\geq 1$.
So, we have $\max\{ 2s-t+\varepsilon, s\} \leq 2s$.  Therefore, it suffices to show the inequality $2s< h$.
Assume that $n=8\ell+r$ where $1\leq r \leq 8$.
Then, by the definition of $s$, we have 
\[
2^{s-1}<8\ell+r\leq 8\ell+8.
\]
Hence, we have 
\[
2^{s-2}< 4\ell +4.
\]
Therefore, we obtain
\[
h\geq 4\ell>2^{s-2}-4\geq  2s
\]
for $s\geq 6$ as required.
\end{proof}

We prove some elementary properties of $D$, say Propositions~\ref{sec22} and \ref{sec23}, which we need in the proof of 
\fullref{sec51}.

\begin{proposition}\label{sec22}
Suppose that $D$ is not empty.
If $k \in D$, then $2^s-k+1\in D$.
\end{proposition}

\begin{proof}
It is easy to see that
\begin{itemize}
\item[(1)] $2^s-k+1\leq n$, 
\item[(2)] $2^s-(2^s-k+1)+1=k\leq n$,
\item[(3)] $\alpha((2^s-k+1)-1)=\alpha(2^s-k)\geq 2$, 
\item[(4)] $\alpha(2^s-(2^s-k+1))=\alpha(k-1)\geq 2$. \qedhere
\end{itemize}
\end{proof}

\begin{proposition}\label{sec23}
Suppose that $D$ is not empty and $k \in D$.
Then\/{\rm:}
\begin{itemize}
\item[{\rm (1)}]
$2^{s-t+1} (k-1)+1 >n$.
\item[{\rm (2)}] 
If $\varepsilon=0$, then $2^{s-t}(k-1)+1>n$.
\end{itemize}
\end{proposition}

\begin{proof}
First, we prove (1).
Since $2^s-k+1$ is also in $D$, we have
\[
2^s-n\leq k-1.
\]
Hence,  we have
\[
2^{s-t+1}(k-1)+1>2^s+2^{s-t+1}+1>n. 
\]
Next, we prove (2).
Since $2^s-k+1$ is also in $D$, by the definition of $m$, we have
\[
2^s-k+1\leq m.
\]
Thus, we have 
\[
2^s-m\leq k-1.
\]
Since $\varepsilon=0$, we have
\[
2^{s-t}(k-1)+1\geq 2^{s-t}(2^s-m)+1=m'>n. \qedhere
\]
\end{proof}

It is clear that the number of integers in $E$ is $n-s-1$.
For $k=0, \ldots, s-t-1$, we define $\sigma(k)$
by
$$\sigma(k)=2^s-2^{s-1-k}-1.$$
\begin{equation*}
C_0=\{ \sigma(k)\;|\; k=0, \ldots, s-t-1\}.\tag*{\rm Let}
\end{equation*}
Then, it is easy to see that $C_0$ is a subset of $E$.
For $k=s-t$, we define $\sigma(k)$ to be $m$ if $\varepsilon=1$.
For $k=s-t+\varepsilon, \ldots, n-s-2$, we define 
$\sigma(k)$ as follows:
\[
\sigma(k) \in \{ a \in E\;|\; a\not \in C_0, a \not = m \;\mbox{if $\varepsilon=1$} \},
\]
and then we have
\[
\sigma(s-t+\varepsilon)<\cdots < \sigma(n-s-2). \]
Let $\tau(k)=2^{s-1}+2^k+1$ for $k=0,\ldots, s-t-1$.
Let $C=C_0\cup C_1$, where
\[
C_1=\{ \;\tau(k) \; | \;k=0,\ldots, s-t-1\;\}.
\]
What we need in the proof of \fullref{sec52} in \fullref{sec5} is the following 
Propositions~\ref{sec24} and \ref{sec25}.
For the rest of this section, we assume that $n\geq 18$, $n \not = 2^{s-1}+1$
and
 $s\geq 5$.
 
 %%%%%%%%%%%%%%%%%%%%%%%%%%%%%%%%%%%%%%%%%%%%%%%%%%%%%%%%%%%%%%%%%%%%%%%%%%%%%%%%%%%%%%%%%%%%%%%%%%%%%%%%%%%

\begin{proposition}\label{sec24}
Suppose that $n \geq 18$ and $n \not = 2^{s-1}+1$. Then, 
the integers $\sigma(k)$, $\tau(k)$ $(k=0,\ldots, s-t-1)$ are distinct from each other.
\end{proposition}

\begin{proof}
If $n \geq 18$, then $s\geq 5$, so that $s-1> 3$.
Since $(s-t)$ integers $\sigma(k)$ ($k=0,\ldots, s-t-1$)  
in $C_0$ are distinct from each other, 
since $(s-t)$ integers $\tau(k)$ ($k=0,\ldots, s-t-1$) 
in $C_1$ are also distinct from each other, and since $\alpha(\sigma(k))=s-1$, $\alpha(\tau(k))\leq 3$,
we have that $C_0\cap C_1=\emptyset$ and that $(2s-2t)$ integers $\sigma(k)$, $\tau(k')$ are distinct from each other where $k, k'\in\{0,\ldots, s-t-1\}$.
\end{proof}

\begin{proposition}\label{sec25}
Suppose that $n \geq 18$ and $n \not = 2^{s-1}+1$. If $\varepsilon=1$, then $m$, $m' \not \in C$.
\end{proposition}

The rest of this section is devoted to proving \fullref{sec25} above.
Firstly, we prove that if $n \geq 18$ and if $n \in C$, then we have $\varepsilon=0$.

\begin{proposition} \label{sec26}
Suppose that $n \geq 18$ and $n\not = 2^{s-1}+1$. If $\varepsilon=1$, then we have $m=n$
and $2^{t-1}+1<2^s-n\leq 2^t+1$.
\end{proposition}

\begin{proof}
We prove this proposition by showing that if $m \not = n$, then we have $\varepsilon=0$.
First, we deal with the case $n=2^s$, $2^s-1$ or $2^s-2$. In this case, 
$t=1$, $m=2^s-3$, $m'=2^{s-1}\cdot 3+1>2^s+1> n$. Thus, we have $\varepsilon=0$.
So, without loss of generality, we may assume that $2^{s-1}+2\leq n \leq 2^{s}-3$ and
so we have
\[
2^{t-1}+1<2^s-n\leq 2^t+1.
\]
Suppose that $m \not = n$. Then, $\alpha(n-1)= 1$ or $\alpha(2^s-n)  = 1$.
The equality $\alpha(n-1)=1$ holds  if and only if $n=2^{s-1}+1$. Hence, $\alpha(2^s-n)=1$.
So we have $2^s-n=2^t$, $m=2^s-2^t-1$ and 
\[
m'=2^{s-t} (2^{t}+1)+1=2^s+2^{s-t}+1>n.
\]
Hence, by definition, we have $\varepsilon=0$.
\end{proof}

\begin{proof}[Proof of \fullref{sec25}]
By \fullref{sec26}, we have $m=n$, 
\begin{gather*}
m'=2^{s-t}(2^s-n)+1\\
2^{t-1}+1< 2^s-n \leq 2^t+1.\tag*{\rm and}
\end{gather*}
If $m \in C$ or if $m'\in C$, then one of the following conditions holds:
\begin{itemize}
\item[{\rm (1)}] $n=2^s-2^{s-1-k}-1$, 
\item[{\rm (2)}] $n=2^{s-1}+2^k+1$,
\item[{\rm (3)}] $2^{s-t}(2^s-n)+1=2^s-2^{s-1-k}-1$, 
\item[{\rm (4)}] $2^{s-t}(2^s-n)+1=2^{s-1}+2^k+1$,
\end{itemize}
where $0\leq k \leq s-t-1$. We prove that it is not the case.

\noindent {\rm Case (1)}\qua We have $2^s-n=2^{s-1-k}+1$. So, we have $t=s-1-k$ and
\[
m'-n=2^{s-t}(2^{t}+1)+1-(2^{s}-2^{s-1-k}-1)>0.
\]
This contradicts the assumption $\varepsilon=1$.

\noindent {\rm Case (2)}\qua We have $2^s-n=2^{s-1}-2^{k}-1$. 
So, one of the following statements holds:
\begin{itemize}
\item[(a)] $t=s-1$, $k<s-2$ or 
\item[(b)] $t=s-2$, $k=s-2$.
\end{itemize}
If $s-t=1$ and $k<s-2$, then 
$m'=2^s-2^{k+1}-1$
and
\[
m'-n=2^{s-1}-2^{k+1}-2^k-2.
\]
If $s-t=2$ and  $k=s-2$, then we have  
\[
m'-n=2^{s-1}-2^k-2.
\]
In both cases, we have $m'-n>0$. This contradicts the assumption $\varepsilon=1$.

\noindent {\rm Case (3)}\qua
We have 
\[
2^s-n = 2^t-2^{(s-1-k)-(s-t)} -2^{1-(s-t)}.
\]
By the definition of $t$, we have that $s-t>0$. Moreover, because of the assumption $k\leq s-t-1$, we have $s-1-k>0$.
Since $2^s-n$ is  an integer, we have $s-t=1$ and $k=0$.  So, we have $2^s-n=2^{s-2}-1$. This contradicts 
the inequality \[
2^{t-1}+1<2^s-n.
\]

\noindent {\rm Case (4)}\qua
We have
\[
2^s-n=2^{t-1}+2^{k-(s-t)}.
\]
Since $2^s-n$ is an integer, we have $k-(s-t)\geq 0$. This contradicts the assumption $0\leq k \leq s-t-1$.

Thus, any of the above four conditions (1), \ldots, (4) does not hold. Hence, we have the desired result.
\end{proof}

\section{\groebner bases and regular sequences}\label{sec3}

In this section, we recall the notion of \groebner bases and regular
sequences.  Let $K$ be a field and let $R=K[x_1, \ldots, x_n]$ be a
polynomial ring over $K$ in $n$ variables $x_1, \ldots, x_n$.

Firstly, we recall the definition of \groebner basis and its
elementary properties.  We refer the reader to text books on \groebner
bases such as Adams and Loustaunau \cite{adams}.  We assume that $R$
has a fixed term order on the set of monomials of $R$.  A term order
is often called a monomial order in literature, see Eisenbud \cite{eisenbud}
for example. It is a total order on the set of monomials such that
for monomials $x$, $y$, $z$:
\[
z<xz<yz
\]
if $x<y$ and $z\not=1$.
Let $f$ be an element in $R$. We denote by $\lp(f)$ the leading power, or the leading monomial, 
of $f$ and
by $\lt(f)$ the leading term of $f$. In the case the coefficient field $K$ is $\bF_2$, the leading term and the leading monomial are the same.
Let $G=\{ g_1, \ldots, g_r\}$ be a finite subset of $R$, where we assume that $g_i$'s are nonzero
and $g_i\not = g_j$ for $i\not=j$.

The subset $G$ is called a \groebner basis if 
each polynomial in the ideal $I=(g_1, \ldots, g_r)$ has the leading term divisible 
by the leading term of $g_k$ for some $g_k\in G$.
A polynomial $f$ is said to reduce to zero modulo $G$ 
if and only if there exist $f_1, \ldots, f_s\in R$ and $i_1, \ldots, i_s\in \{ 1,\ldots, r\}$ such
that
\[
f=\sum_{k=1}^{s}  f_kg_{i_k},
 \]
 where
a scalar multiple of $\lp(f_1)\lp(g_{i_1})$ is a nonzero term in $f$,
and for $k=2, \ldots, s$, 
a scalar multiple of $\lp(f_k)\lp(g_{i_k})$ is a nonzero term of
\[
\lp(f-\sum_{\ell=1}^{k-1} f_{\ell}g_{i_\ell}).
\]
It is clear from the definition of \groebner basis that 
when $G=\{g_1, \ldots, g_r\}$ is a \groebner basis, 
a polynomial in $R$ is in the ideal $(g_1, \ldots, g_r)$ 
if and only if $f$ reduces to zero modulo $G$.

The following theorem is known as the Buchberger criterion.
\begin{theorem}[Buchberger]
Let $G=\{g_1, \ldots, g_r\}$ be a finite subset of $R$.
Let \[
S(g_i, g_j)=\frac{\lcm(\lp(g_i),\lp(g_j))}{\lt(g_i)}g_i-\frac{\lcm(\lp(g_i),\lp(g_j))}{\lt(g_j)}g_j,
\]
where $\lcm$ stands for the least common multiple.
The set $G$ is a \groebner basis if  and only if
all $S(g_i, g_j)$ ($i\not = j$)  reduce to zero modulo $G$.
\end{theorem}

\begin{proof}
See the proof of Theorem 1.7.4 in \cite{adams}.
\end{proof}

We also recall the lemma below. 
\begin{lemma}
Let $g_1$,$g_2\in R$ and suppose that both are nonzero. Let $d=\gcd(g_1, g_2)$.
The following statements are equivalent\/{\rm :}
\begin{itemize}
\item[{\rm (1)}] $\displaystyle \lp(\frac{g_1}{d})$ and
 $\displaystyle \lp(\frac{g_2}{d})$ are relatively prime\/{\rm;}
\item[{\rm (2)}] $S(g_1, g_2)$ reduces to zero modulo $\{g_1, g_2\}$.
\end{itemize}
\end{lemma}

\begin{proof}
See the proof of Lemma 3.3.1 in \cite{adams}.
\end{proof}

As an application of this lemma, by the Buchberger criterion, we have the following proposition.
\begin{proposition}\label{sec33}
Let $G=\{ g_1, \ldots, g_r\}$ be a finite set of polynomials in $R$. Suppose that 
the leading terms of $g_i$ and $g_j$ are relatively prime for $i\not = j$. Then, 
the set $G$ is a \groebner basis.
\end{proposition}

%regular sequence

Secondly, we recall the definition of a regular sequence.
A sequence $g_1, \ldots, g_r$ of polynomials in $R$ is called a regular sequence if 
the multiplication by $g_{k}$ induces a  monomorphism 
\[
R \stackrel{\times g_1}{\longrightarrow} R
\]
for $k=1$ and a monomorphism
\[
R/(g_1, \ldots, g_{k-1}) \stackrel{\times g_k}{\longrightarrow}  R/(g_1, \ldots, g_{k-1})
\]
for $k=2, \ldots, r$.
If $g_1, \ldots, g_r$ are homogeneous polynomials, 
then the \poincare series of $R/(g_1, \ldots, g_r)$ is given by
\[
\prod_{k=1}^r {\left(1-t^{\deg g_k}\right)} \left/ 
\prod_{k=1}^n {\left(1-t^{\deg x_k}\right)} \right..
\]
We need the following lemma in the proof of \fullref{sec52} in \fullref{sec5}.

\begin{lemma}\label{sec34}
Suppose that $g_1,\ldots, g_r$ are polynomials in $R$
such that the leading monomials of $g_i$ and $g_j$ are relatively prime for $i\not = j$.
Then,  the sequence $g_1, \ldots, g_r$ is a regular sequence.
\end{lemma}

\begin{proof}
Since $R$ is an integral domain, it is clear that the multiplication by $g_1$ induces a monomorphism
\[
R \to R.
\]
For $k=2, \ldots, r$, 
by \fullref{sec33}, $\{g_1, \ldots, g_{k-1}\}$ is a \groebner basis for $k=2, \ldots, r$.
Suppose that
 $f\not \in (g_1, \ldots, g_{k-1})$ and that $g_kf\in (g_1, \ldots, g_{k-1})$.
 Without loss of generality, we may assume that the leading term of $f$ is not divisible by 
$\lp(g_i)$ where $i=1, \ldots, k-1$ and that
the leading term $\lp(g_k)\lp(f)$ of $g_kf$ is divisible by some $\lp(g_i)$ where $i \in \{1, \ldots, k-1\}$.
Since $\lp(g_i)$ and $\lp(g_k)$ are relatively prime in $R$, we see that 
$\lp(f)$ is divisible by $\lp(g_i)$.
It is a contradiction. Thus, we have that if $g_kf \in (g_1, \ldots, g_{k-1})$, then 
$f\in (g_1, \ldots, g_{k-1})$.
\end{proof}

\section{Steenrod squares and the change-of-rings spectral\newline sequence}\label{sec4}

In this section, we recall some facts on the action of Steenrod squares on cotorsion products
and spectral sequences. We refer the reader to Singer's book \cite{singer}.

Firstly, we recall the action of the Steenrod squares on the cotorsion product\break $\cotor_A(\bF_2,\bF_2)$ for a connected Hopf algebra $A$ over $\bF_2$.
Let 
\[
\phi\co A\to A\otimes A
\]
 be the coproduct of $A$.
Let $\bar{A}$ be the submodule generated by the positive degree elements.
We denote by 
\[
\bar{\phi}\co \bar{A}\to \bar{A}\otimes \bar{A}
\]
the reduced coproduct.
The cotorsion product $\cotor_{A}(\bF_2,\bF_2)$  is a graded $\bF_2$--algebra
 generated by elements $ [x_1|\cdots |x_r]$ where 
we denote by $ [x_1|\cdots | x_r]$ the element represented by $x_1\otimes \cdots \otimes x_r \in  \bar{A}\otimes \cdots \otimes \bar{A}$.

\fullref{sec41} below is  a variant of Proposition~{1.111}  in Singer's book \cite{singer}. 
The unstable condition below immediately follows from the definition and the construction of Steenrod squares in \cite{singer}.
It is also called Steenrod Operation Theorem~{A1.5.2} in Ravenel \cite{ravenel}, which is a re-indexed form of 11.8 of May \cite{may}.

\begin{theorem} \label{sec41}
With the notation above, for $p\geq 0$, $k\geq 0$, there exist homomorphisms 
\[ 
\Sq^k\co \cotor_A^{p}(\bF_2,\bF_2)\to \cotor_A^{p+k} (\bF_2,\bF_2)
\]
satisfying

{\rm (1)}\qua the unstable condition\/{\rm:}
\[
\begin{array}{ll}
\Sq^0[x]=[x^2], \\
\Sq^1[x]=[x|x]=[x]^2, \\
\Sq^k[x]=0 & \mbox{for $k\geq 2$\/{\rm;} }
\end{array}
\]
{\rm (2)}\qua the Cartan formula\/{\rm:} 
\[
\Sq^k( xy)=\sum_{i+j=k, i,j \geq 0} (\Sq^i x) (\Sq^j y).
\]
\end{theorem}
Note that $\Sq^0\co \cotor_A^p (\bF_2, \bF_2)\to \cotor_A^p(\bF_2, \bF_2)$ is not the identity homomorphism.

Secondly, we recall the action of the Steenrod squares on the change-of-rings spectral sequence.
Let us consider an extension of connected Hopf algebras:
\[
\Gamma \to A \to \Lambda.
\]
Then, there exists the change-of-rings spectral sequence $$\{ E_r^{p,q}, d_r\co E_r^{p,q}\to E^{p+r,q-r+1}_r\}$$
with the $E_2$--term
\[
E_2^{p,q}=\cotor^p_\Gamma(\bF_2, \cotor^q_A(\Gamma, \bF_2)).
\]
It converges to the cotorsion product $\cotor_A(\bF_2,\bF_2)$ and
 is a first quadrant cohomology spectral sequence of graded $\bF_2$--algebras. 

The following is a combined form of Theorems~{2.15} and {2.17} in Singer's book \cite{singer}.
\begin{theorem} \label{sec42}
With the notation above,
for all $p,q\geq 0$, $r\geq 2$, there exist homomorphisms
\[
\begin{array}{ll}
\Sq^k\co E_r^{p,q}\to E_r^{p,q+k} & \mbox{if $0\leq k \leq q$,} \\
\Sq^k\co E_r^{p,q}\to E_{r+k-q}^{p+k-q,2q} & \mbox{if $q\leq k \leq q+r-2$,} \\
\Sq^k\co E_r^{p,q}\to E_{2r-2}^{p+k-q,2q} & \mbox{if $q+r-2\leq k $,}
\end{array}
\]
such that

{\rm (1)}\qua 
if $\alpha \in E_r^{p,q}$, then
both $\Sq^k \alpha$ and $\Sq^k d_r\alpha$ survive to $E_t$, where
\[
\begin{array}{ll}
t=r & \mbox{if $0\leq k \leq q-r+1$,} \\
t=2r+k-q-1 & \mbox{if $q-r+1\leq k \leq q$,}\\
t=2r-1 & \mbox{if $q\leq k$\/{\rm;}}
\end{array}
\]
{\rm(2)}\qua in $E_t$, we have
\[
d_t(\Sq^k\alpha)=\Sq^k d_r\alpha;
\]
{\rm (3)}\qua at the $E_\infty$--level, $\Sq^k$ is compatible with the action of $\Sq^k$ on $\cotor_A(\bF_2, \bF_2)$, that is,
if we denote by
\[
\pi_{p,q}\co F^p\cotor_A^{p+q}(\bF_2, \bF_2)\to E_\infty^{p,q}
\]
the edge homomorphism, then\/{\rm:}
\[
\begin{array}{ll}
\Sq^k \pi_{p,q}=\pi_{p,q+k} \Sq^k & \mbox{ for $k\leq q$ and}\\
\Sq^k \pi_{p,q}=\pi_{p+k-q,2q} \Sq^k & \mbox{ for $k\geq q$,}
\end{array}
\]
where the $\Sq^k$ in the right hand-side of the above equalities are the one given in \fullref{sec41}.
\end{theorem}

\section{Cotorsion products}\label{sec5}

We refer the reader to  the book of Mimura and Toda \cite{mimuratoda}, Mimura \cite{mimura} and their references
for  the cohomology of compact Lie groups.  Recall that the mod $2$ cohomology of $\spin(n)$ is given as follows:
Let $E$ be the set $E$ defined in \fullref{sec2}. 
Let $\Delta$ be an algebra generated by $x_k$ with the relation $x_k^2=x_{2k}$ where $x_k=0$ if $k+1\not\in E$.
As an algebra over $\bF_2$, we have
\[
H^*(\spin(n);\bF_2)=\Delta\otimes  \mbox{\mathitt\char'003}( y_{2^s-1}).
\]
The reduced coproduct $\bar{\phi}$ is given by
\[
\bar{\phi}(x_{k})=0
\]
for $k+1\in E$ 
and 
\[
\bar{\phi}(y_{2^s-1})=\sum_{i+j=2^{s-1}} x_{2i}\otimes x_{2j-1}.
\]
In this section, by computing the change-of-rings spectral sequence associated with 
the extension of Hopf algebras:
\[
\Delta \to H^{*}(\spin(n);\bF_2) \to \mbox{\mathitt\char'003}(y_{2^s-1}),
\]
we prove \fullref{sec12}. 
The subalgebra $\Delta$ is the image of the induced homomorphism 
\[
\pi^*\co H^{*}(SO(n);\bF_2) \to H^{*}(\spin(n);\bF_2).
\]
The $E_2$-term of the spectral sequence is given by 
\[
\mathrm{Cotor}_{\Delta}(\bF_2, \mathrm{Cotor}_{H^{*}(\spin(n);\bF_2)}(\Delta, \bF_2)).
\]
We call this spectral sequence the change-of-rings spectral sequence.
As a matter of fact, it is nothing but the change-of-coalgebras spectral 
sequence in Section 2 of Moore and Smith \cite{mooresmith}. It is also noted in \cite{mooresmith} that the $E_2$--term is isomorphic to 
\[
\mathrm{Cotor}_{\Delta}(\bF_2, \bF_2) \otimes \mathrm{Cotor}_{\mbox{\mathits\char'003}(y_{2^s-1})}(\bF_2, \bF_2).
\]
For the sake of notational simplicity, let 
\begin{gather*}
A=H^*(\spin(n);\bF_2)\\
 B=H^{*}(SO(n);\bF_2).\tag*{\rm and}\end{gather*}
Firstly, we collect some results on $\cotor_B(\bF_2, \bF_2)$
 and the Rothenberg--Steenrod spectral sequence for the mod 2 cohomology of $BSO(n)$.
As an algebra, $B$ is generated by $x_i$ 
with the relations $x_i^2=x_{2i}$ where $x_i=0$ for $i\geq n$.
As a coalgebra, $x_i$ ($i=1, \ldots, n-1$) are primitive and $B$ is primitively generated.
So, the cotorsion product $\cotor_B(\bF_2, \bF_2)$ is a polynomial algebra
$
\bF_2[w_2 ,\ldots, w_n]
$
where $w_{k+1}$ is represented by $[x_{k}]\in \cotor^{1,k}_B(\bF_2, \bF_2)$.
It is also clear that the Rothenberg-Steenrod spectral sequence collapses at the $E_2$--level 
and hence
we have $H^{*}(BSO(n);\bF_2)=\bF_2[w_2, \ldots, w_n]$, where, by abuse of notation, 
we denote by $w_{k+1}$ the element in $H^{*}(BSO(n);\bF_2)$ represented by 
\[
w_{k+1}\in E_{\infty}^{1,k}=E_2^{1,k}
=\cotor_B^{1,k}(\bF_2, \bF_2).
\]
Let $v_0=w_2 \in \cotor_B(\bF_2, \bF_2)$.
For $1\leq k\leq s-1$, let \[
v_k=\underbrace{\Sq^{0}\cdots \Sq^0}_{\mbox{$k$-times}} v_0 \in \cotor_B(\bF_2, \bF_2).
\]
By the unstable condition in \fullref{sec41}, 
we have $v_k=w_{2^k+1}$.
$$
v_s=\sum_{i+j=2^{s-1}} w_{2i+1}w_{2j},\leqno{\rm Let}
$$
where we assume that $i,j\geq 0$ and $w_0=w_1=0$ and $w_i=0$ for $i>n$.
We define an element $v_{s+k}$ in $\cotor_B(\bF_2, \bF_2)$  for $k\geq 1$ by
$$v_{s+k}=\Sq^{2^{k-1}}\cdots \Sq^1 v_s.$$
\begin{gather*}
R=\bF_2[w_2, \ldots, w_n]/(v_0,\ldots, v_{s-1})\tag*{\rm Let}
\end{gather*}
be the polynomial ring generated by variables $w_k$
where $k$ ranges over the set $E$. This is isomorphic to the cotorsion product $\mathrm{Cotor}_{\Delta}(\bF_2, \bF_2)$.

We have the following proposition.

\begin{proposition}\label{sec51}\ \

%The following hold{\rm :}
\begin{itemize}
\item[{\rm (1)}] The polynomial $v_{2s-t+1}$ is zero  in $R$.
\item[{\rm (2)}]
If $\varepsilon=0$, then the polynomial $v_{2s-t}$ is also zero in $R$.
\end{itemize}
\end{proposition}

\begin{proof}
Suppose that $w_iw_j$ is a nonzero term in $v_s$.
By definition, it is easy to see that both $i$ and $j$ are in $D$.
By the unstable condition and by the Cartan formula in \fullref{sec41}
for $k\geq 1$, we have
\[
\Sq^{2^{k-1}}\cdots \Sq^1 w_iw_j=
w_i^{2^k}w_{2^{k}(j-1)+1}
+w_{2^k(i-1)+1}w_{j}^{2^k}.
\]
By \fullref{sec23}, we have 
\[
\Sq^{2^{k-1}}\cdots \Sq^1 w_iw_j=0
\]
in the case  $k\geq s-t$ or in the case $\varepsilon=0$ and $k=s-t-1$. 
\end{proof}

To prove \fullref{sec12}, we need the following result.

\begin{proposition}\label{sec52}
If $n\geq 9$ and if $n \not = 2^{s-1}+1$, then
the sequence $v_s, \ldots, v_{h'-1}$ is a regular sequence in $R$.
\end{proposition}

\begin{proof}
Firstly,  we deal with the case $10 \leq n \leq 16$.
In this case, $s=4$ and we have
\[
\begin{array}{lll}
v_4=w_7w_{10}+w_6w_{11}+w_4w_{13},&
v_5=w_{13}w_{10}^2+w_{11}^3+w_{7}w_{13}^2, &
v_6=w_{13}^5,
\end{array}
\]
where $w_i=0$ for $n<i\leq 16$.
We consider the degree reverse lexicographic order such that 
\[
w_{4}>w_{6}>w_{7}>w_{8}>w_{10}>w_{11}>w_{12}>w_{13}>w_{14}>w_{15}>w_{16}.
\]
For $n=13$, $14$, $15$, $16$, we have $t=1$ and $h'=7$
and the leading terms of $v_4$, $v_5$, $v_6$ are
$w_7w_{10}$, $w_{11}^3$, $w_{13}^5$, respectively. 
So, by \fullref{sec34}, we have the desired result.
For $n=11$, $12$, we have $t=2$, $h'=6$ and the leading terms of $v_4=w_{7}w_{10}+w_{6}w_{11}$,
 $v_5=w_{11}^3$ are
$w_7w_{10}$, $w_{11}^3$, respectively. So, by \fullref{sec34}, we have the desired result.
For $n=10$, we have $t=3$, $h'=5$ and 
it is clear that the sequence $v_4=w_7w_{10}$ is a regular sequence.

Next, we deal with the case $s\geq 5$, $n \not = 2^{s-1}+1$. In order to use \fullref{sec34}, 
we need to  define the term order on the set of monomials in $R$ as follows:
Suppose that
\[
x=w_{\sigma(0)}^{i_0}\cdots w_{\sigma(n-s-2)}^{i_{n-s-2}}, \;\;\;
y=w_{\sigma(0)}^{j_0}\cdots w_{\sigma(n-s-2)}^{j_{n-s-2}}.
\]
We define the weight of $x$ by
\[
w(x)=\sum_{\ell=0}^{s-t+\varepsilon-1} i_{\ell}.
\]
We say $x>y$ if
\begin{itemize}
\item[{\rm (1)}] $w(x)> w(y)$ or
\item[{\rm (2)}] $w(x)=w(y)$ and there is an integer $k$ such that $i_{\ell}=j_{\ell}$ for $\ell<k$
and $i_k>j_k$.
\end{itemize}

Since $2^{k}(2^s-\sigma(\ell))+1>n$ for $\ell<k$, we have $w_{2^k(2^s-\sigma(\ell))+1}=0$ for $\ell<k$. So, we obtain 
\[
v_{s+k}\equiv \sum_{\ell=k}^{s-t+\varepsilon-1} w_{\sigma(\ell)}^{2^k}w_{2^{k}(2^s-\sigma(\ell))+1}
\]
modulo terms with weight less than $2^k$.
The leading terms of $v_s, \ldots, v_{2s-t-1}$ are 
$w_{\sigma(0)}w_{\tau(0)},\ldots, w_{\sigma(s-t-1)}^{2^{s-t-1}}w_{\tau(s-t-1)}$
and the leading term of $v_{2s-t}$ is $w_{m}^{2^{s-t}}w_{m'}$ if $\varepsilon=1$.
By \fullref{sec24}, we have
\[
\gcd(w_{\sigma(k)}^{2^k}w_{\tau(k)}, w_{\sigma(k')}^{2k'}w_{\tau(k')})=1
\]
for $k\not = k' \in C_0$ and, 
by \fullref{sec25}, we have
\[
\gcd(w_{\sigma(k)}^{2^k}w_{\tau(k)}, w_m^{2^{s-t}}w_{m'})=1
\]
for $k \in C_0$ when $\varepsilon=1$.
Therefore, by \fullref{sec34}, we have that the sequence $v_s, \ldots, v_{2s-t+\varepsilon-1}$ is a regular sequence.
\end{proof}

By abuse of notation, we identify the above
\[
R=H^*(BSO(n);\bF_2)/(v_0,\ldots, v_{s-1})=\mathrm{Cotor}_{\Delta}(\bF_2,\bF_2)
\]
 with the image of 
\[
B\pi^*\co H^*(BSO(n);\bF_2)\to H^*(B\spin(n);\bF_2)
\]
 and with $E_2^{*,0}$ in the change-of-rings spectral sequence.
 Thus, we have
 \[
 E_2^{*,*}=R \otimes \bF_2[\zeta], \]
 where $\zeta\in E_2^{0,1}$ is the element represented by $[y_{2^s-1}]$.
Now, we complete the proof of \fullref{sec12}.

\begin{proof}[Proof of Theroem~\ref{sec12}]
Let us consider the cobar resolution 
\[
\bar{A} \stackrel{d}{\longrightarrow} \bar{A}\otimes \bar{A} \stackrel{d}{\longrightarrow} \bar{A}\otimes \bar{A}\otimes \bar{A} \rightarrow \cdots.
\]
It is clear that 
\[
d(y_{2^s-1})=\sum_{i+j=2^{s-1}} x_{2i}\otimes x_{2j-1}
\]
and so the element
\[
v_s=\sum_{i+j=2^{s-1}} w_{2i+1}w_{2j}
\]
 is zero in $\cotor_{A}(\bF_2, \bF_2)$. 
 Therefore, $v_s \in E_2^{2,0}$ is equal to $d_2(\zeta)$.
Hence, by \fullref{sec42}, we have that
both $\Sq^{2^{k-1}}\cdots \Sq^1 \zeta
\in E_{2}^{0, 2^k} $ and $\Sq^{2^{k-1}}\cdots \Sq^1 d_2 \zeta \in E_2^{2^k+1, 0}$ survive to the $E_{2^k+1}$--term and 
\[
d_{2^k+1}\Sq^{2^{k-1}}\cdots \Sq^1 \zeta=\Sq^{2^{k-1}}\cdots \Sq^1 d_2\zeta \in E_{2^k+1}^{0,2^k}.
\]
For $k=1, \ldots, h'-s-1$, we have, by the unstable condition, 
\[
\Sq^{2^{k-1}}\cdots \Sq^1 \zeta=\zeta^{2^k}
\]
and, by definition,  \[
\Sq^{2^{k-1}}\cdots \Sq^1 d_2\zeta=v_{s+k}.
\]
Since $v_s, \ldots, v_{h'-1}$ is a regular sequence in $R$
and since $E_2=R\otimes \bF_2[\zeta]$, 
we have, for $k=1, \ldots, h'-s-1$, 
\[
E_{2^k+1}=\cdots =E_{2^{k-1}+2}=R/(v_s, \ldots, v_{s+k-1}) \otimes \bF_2[\zeta^{2^{k}}].
\]
Moreover, we have
\[
E_{\infty}=E_{2^{h'-s-1}+2}=R/(v_s,\ldots, v_{h'-1}) \otimes \bF_2[\zeta^{2^{h'-s}}].
\]
It is clear that an algebra  homomorphism 
\[
\varphi\co H^*(BSO(n);\bF_2) \otimes \bF_2[z'] \rightarrow
H^{*}(B\spin(n);\bF_2)
\]
defined by $\varphi(w_k\otimes 1)=B\pi^*(w_k)$ and $\varphi(1\otimes z')=z''$, where 
$z''$ represents $\zeta^{2^{h'-s}}\in E_{\infty}^{0,2^{h'-s}}$,
induces an isomorphism
\[
R \otimes \bF_2[z']/(v_s\otimes 1, \ldots, v_{h'-1}\otimes 1) \to \cotor_A(\bF_2, \bF_2).
\]
So there is no extension problem and 
it completes the proof of \fullref{sec12}.
\end{proof}

\bibliographystyle{gtart}
\bibliography{link}

\end{document}